\documentclass{amsart}
\usepackage{amsmath,amsfonts}
\usepackage{hyperref}

\title[Determinantal Representations: Past, Present, and Future]
{LMI Representations of Convex Semialgebraic Sets and Determinantal Representations of Algebraic Hypersurfaces:
Past, Present, and Future}

\author[V. Vinnikov]{Victor Vinnikov}
\address{Department of Mathematics \\
Ben-Gurion University of the Negev\\
Beer-Sheva, Israel, 84105}
\email{vinnikov@math.bgu.ac.il}
\dedicatory{To Bill Helton, on the occasion of his 65th birthday}

\newcommand{\trans}{\top}
\newcommand{\sym}{{\mathbb S}}
\newcommand{\herm}{{\mathbb H}}
\newcommand{\mat}[2]{\ensuremath{{#1}^{#2\times #2}}}
\newcommand{\rmat}[3]{\ensuremath{{#1}}^{#2\times #3}}

\newcommand{\PP}{{\mathbb P}}
\newcommand{\NN}{{\mathbb N}}
\newcommand{\ZZ}{{\mathbb Z}}
\newcommand{\RR}{{\mathbb R}}
\newcommand{\CC}{{\mathbb C}}

\newcommand{\cB}{{\mathcal B}}
\newcommand{\cC}{{\mathcal C}}
\newcommand{\cE}{\boldsymbol{\mathcal E}}
\newcommand{\cI}{{\mathcal I}}
\newcommand{\cK}{\boldsymbol{\mathcal K}}
\newcommand{\cL}{\boldsymbol{\mathcal L}}
\newcommand{\cO}{\boldsymbol{\mathcal O}}
\newcommand{\cV}{{\mathcal V}}
\newcommand{\bomega}{\boldsymbol{\omega}}
\newcommand{\yashar}{{\mathcal L}}

\newcommand{\tcE}{\widetilde{\cE}}
\newcommand{\tcL}{\widetilde{\cL}}
\newcommand{\tcV}{\widetilde{\cV}}

\newcommand{\qut}{{\tilde q}}

\newcommand{\Int}{\operatorname{Int}}
\newcommand{\adj}{\operatorname{adj}}
\newcommand{\coker}{\operatorname{coker}}
\newcommand{\Proj}{\operatorname{Proj}}

\newcommand{\sing}{\text{sing}}

\theoremstyle{plain}
\newtheorem{thm}{Theorem}[section]

\newtheorem{lem}[thm]{Lemma}
\newtheorem{prop}[thm]{Proposition}
\newtheorem{conj}[thm]{Conjecture}

\theoremstyle{definition}

\theoremstyle{remark}

\numberwithin{equation}{section}

\begin{document}

\begin{abstract}
10 years ago or so Bill Helton
introduced me to some mathematical problems arising from semidefinite programming.
This paper is a partial account of what was and what is happening with one of these problems,
including many open questions and some new results.
\end{abstract}

\maketitle

\section{Introduction}
\label{sec:intro}

Semidefinite programming (SDP) is probably the most important new development in optimization in the last two decades.
The (primal) semidefinite programme is to minimize an affine linear functional $\ell$ on $\RR^d$ 
subject to a linear matrix inequality (LMI) constraint
\begin{equation*}
A_0 + x_1 A_1 + \cdots + x_d A_d \geq 0;
\end{equation*}
here $A_0,A_1,\ldots,A_d \in \sym\mat{\RR}{n}$ (real symmetric $n \times n$ matrices) for some $n$ and $Y \geq 0$
means that $Y \in \sym\mat{\RR}{n}$ is positive semidefinite (has nonnegative eigenvalues or equivalently satisfies
$y^\trans Y y \geq 0$ for all $y \in \RR^n$). 
This can be solved efficiently, both theoretically (finding an approximate solution with a given accuracy $\epsilon$
in a time that is polynomial in $\log(1/\epsilon)$ and in the input size of the problem)
and in many concrete situations,
using interior point methods.
Notice that semidefinite programming is a far reaching extension of linear programming
(LP) which corresponds to the case when the real symmetric matrices $A_0,A_1,\ldots,A_d$ commute (i.e., are simultaneously
diagonalizable). 
The literature on the subject is quite vast, and we only mention 
the pioneering book \cite{NN94}, the surveys \cite{VB96} and \cite{Nem06}, 
and the book \cite{SIG97} for applications to systems and control. 

One very basic mathematical question 
is which convex sets arise as feasibility sets for SDP? In other words, 
{\em given a convex set $\cC$,
do there exist $A_0,A_1,\ldots,A_d \in \sym\mat{\RR}{n}$ for some $n$ such that}
\begin{equation} \label{eq:lmirep}
\cC = \left\{x=(x_1,\ldots,x_d) \in \RR^d \colon 
A_0 + x_1 A_1 + \cdots + x_d A_d \geq 0\right\}?
\end{equation}
We refer to \eqref{eq:lmirep} as a {\em LMI representation} of $\cC$
\footnote{
We can also consider a (complex) self-adjoint LMI representation of $\cC$, meaning
that $A_0,A_1,\ldots,A_d \in \herm\mat{\CC}{n}$ (complex hermitian $n \times n$ matrices) for some $n$.
If $A = B + iC \in \herm\mat{\CC}{n}$ with $B,C \in \mat{\RR}{n}$, and we set
$\tilde A = \left[\begin{smallmatrix} B & -C \\ C & B \end{smallmatrix}\right] \in \sym\mat{\RR}{2n}$,
then $A \geq 0$ if and only  $\tilde A \geq 0$ and $\det \tilde A = (\det A)^2$.
So a self-adjoint LMI representation gives a real symmetric LMI representation
as defined in the main text with the size of matrices doubled and the determinant of
the linear matrix polynomial squared, see \cite[Section 1.4]{RG95} and \cite[Lemma 2.14]{NT1}.}\label{foot:herm}.
Sets having a LMI representation are also called {\em spectrahedra}.
This notion was introduced and studied in \cite{RG95}, and the above question ---
which convex sets admit a LMI representation, i.e., are spectrahedra --- was formally posed in \cite {PS03}.
A complete answer for $d=2$ was obtained in \cite{HV07}, though there are still outstanding computational questions,
see \cite{Hen10,PSV1,PSV2};
for $d>2$, no answer is known,
though the recent results of \cite{Br1,NT1,NPT1} shed some additional light on the problem. 
It is the purpose of this paper to survey some aspects of the current state of the affairs.

Since a real symmetric matrix is positive semidefinite if and only if all of its principal minors are nonnegative,
the set on the right-hand side of \eqref{eq:lmirep} coincides with the set where all the principal minors of
$A_0 + x_1 A_1 + \cdots + x_d A_d$ are nonnegative. Therefore if a convex set $\cC$ admits a LMI representation 
then $\cC$ is a {\em basic closed semialgebraic set} 
(i.e., a set defined by finitely many nonstrict polynomial inequalities). 
However, as shown in \cite{HV07}, $\cC$ is in fact much more special: it is a {\em rigidly convex algebraic interior},
i.e., an {\em algebraic interior} whose {\em minimal defining polynomial} satisfies the {\em real zero} ($RZ$)
condition with respect to any point in the interior of $\cC$.
Furthermore, LMI representations are (essentially) 
{\em positive real symmetric determinantal representations} of certain multiples of the minimal
defining polynomial of $\cC$. 
This reduces the question of the existence (and a construction) of LMI representations to an old problem of algebraic
geometry --- we only mention here the classical paper \cite{D1900} and refer to \cite{Bea00}, \cite[Chapter 4]{Dol1},
and \cite{KV1} 
for a detailed bibliography ---
but with two additional twists: first, we require positivity; second, there is a freedom provided by allowing multiples
of the given polynomial.

This paper is organized as follows. In Section \ref{sec:lmitodetrep} we define rigidly convex sets and $RZ$ polynomials,
and explain why LMI representations are determinantal representations. 
In Section \ref{sec:resultsdetrep} we discuss 
some of what is currently known and unknown about determinantal representations, with a special emphasis on
positive real symmetric determinantal representations. 
In Section \ref{sec:sheaves} 
we review some of the ways to (re)construct a determinantal representation starting from its kernel sheaf, 
especially the construction of the adjoint matrix of a determinantal representation that goes back to \cite{D1900}
and was further developed in \cite{Vin89,BV96,KV1}.
In Section \ref{sec:interlace} we show how this construction yields positive self-adjoint determinantal representations
in the case $d=2$ by using a $RZ$ polynomial that {\em interlaces} the given $RZ$ polynomial.
This provides an alternative proof of the main result of \cite{HV07} (in a slightly weaker form since we obtain
a representation that is self-adjoint rather than real symmetric) which is constructive algebraic in that it avoids
the use of theta functions.

We have concentrated in this paper on the non-homogenous setting (convex sets) rather than on the homogeneous setting
(convex cones). In the homogeneous setting, $RZ$ polynomials correspond to {\em hyperbolic polynomials}
and rigidly convex algebraic interiors correspond to their {\em hyperbolicity cones},
see, e.g.,  \cite{Ga51,Ga59,Lax58,Nui68,Gu97,BGLS01,Re06}.
Theorem \ref{thm:2d} then provides a solution the Lax conjecture concerning homogeneous hyperbolic
polynomials in three variables, see \cite{LPR05}, whereas Conjecture \ref{conj:gend}, which may be called
the generalized Lax conjecture, states that {\em any hyperbolicity cone is a semidefinite slice}, i.e., 
equals the intersection
of the cone of positive semidefinite matrices with a linear subspace.

Finally, the LMI representation problem considered here is but one of the several important problems of this kind
arising from SDP. Other major problems have to do with lifted LMI representations (see \cite{Las09,HN09,HN10})
and with the free noncommutative setting (see \cite{HMcCPV09,HMcC1}).

\subsection*{Acknowledgments}
Apart from my joint work with Bill Helton, a lot of what is described here is based 
on earlier joint work with Joe Ball,
as well as on more recent collaboration with Dmitry Kerner.
It is a pleasure to thank Didier Henrion, Tim Netzer, Daniel Plaumann, 
and Markus Schweighofer for many useful discussions.   

\section{From LMI representations of convex sets to determinantal representations of polynomials}
\label{sec:lmitodetrep}

\subsection{}

A closed set $\cC$ in ${\mathbb R}^d$  is called an {\em algebraic interior} \cite[Section 2.2]{HV07}
if there is a polynomial $p \in \RR[x_1,\ldots,x_d]$
such that  $\cC$ equals the closure of
a connected component of
$$\{x \in {\mathbb R}^d \colon p(x) > 0\}.$$
In other words, there is a $p \in \RR[x_1,\ldots,x_d]$
which vanishes on the boundary $\partial \cC$ of ${\mathcal C}$
and such that
$\{ x \in \cC \colon p(x) > 0\}$ is connected with closure equal 
to $\cC$.
(Notice that in general $p$ may vanish also at some points in the interior
of $\cC$; for example, look at $p(x_1,x_2) = x_2^2 - x_1^2 (x_1 - 1)$.)
We call $p$ a defining polynomial of $\cC$.
It is not hard to show that if $\cC$ is an algebraic interior
then a minimal degree defining polynomial $p$ of $\cC$
is unique (up to a multiplication by a positive constant);
we call it a {\em minimal defining polynomial} of $\cC$,
and it is simply a reduced (i.e., without multiple irreducible factors)
polynomial such that the real affine hypersurface
\begin{equation} \label{eq:real_affine_hyper}
\cV_p(\RR) = \{x \in \RR^d \colon p(x) = 0\}
\end{equation} 
equals the Zariski closure
$\overline{\partial \cC}^{\text{Zar}}$ of the boundary $\partial \cC$ in $\RR^d$
(normalized to be positive at an interior point of $\cC$).
Any other defining polynomial $q$ of $\cC$
is given by $q=ph$ where $h$ is an arbitrary polynomial which
is strictly positive on a dense connected subset of ${\mathcal C}$.
An algebraic interior is a semialgebraic set (i.e., a set defined by a finite boolean combination
of polynomial inequalities) since it is the closure of a connected component of a semialgebraic set.

Let now $\cC$ be a convex set in $\RR^d$ that admits a LMI representation \eqref{eq:lmirep}.
We will assume that $\Int\cC \neq \emptyset$; it turns out that by restricting the LMI representation
(i.e., the matrices $A_0,A_1,\ldots,A_d$) to a subspace of $\RR^n$,
one can assume without loss of generality that 
$A_0 + x_1 A_1 + \cdots + x_d A_d > 0$ for one and then every point of $\Int\cC$
($Y > 0$
means that $Y \in \sym\mat{\RR}{n}$ is positive definite, i.e.,  $Y$ has strictly positive eigenvalues or 
equivalently satisfies
$y^\trans Y y > 0$ for all $y \in \RR^n$, $y \neq 0$).
It is then easy to see that $\cC$ is an algebraic interior with defining polynomial
$\det(A_0 + x_1 A_1 + \cdots + x_d A_d)$.
Conversely, if $\cC$ is an algebraic interior with defining polynomial
$\det(A_0 + x_1 A_1 + \cdots + x_d A_d)$, and $A_0 + x_1 A_1 + \cdots + x_d A_d > 0$
for one point of $\Int \cC$, then it follows easily that \eqref{eq:lmirep} is a LMI representation
of $\cC$. (See \cite[Section 2.3]{HV07} for details.)

Let $q(x) = \det(A_0 + x_1 A_1 + \cdots + x_d A_d)$, let $x^0=(x^0_1,\ldots,x^0_d) \in \Int\cC$,
and let us normalize the LMI representation by $A_0 + x^0_1 A_1 + \cdots + x^0_d A_d = I$.
We restrict the polynomial $q$ to a straight line through $x^0$, i.e., for any $x \in \RR^d$ we consider
the univariate polynomial $q_x(t) = q(x^0 + tx)$. Because of our normalization, we can write
$$
q_x(t) = \det(I + t(x_1 A_1 + \cdots + x_d A_d)),
$$
and since all the eigenvalues of the real symmetric matrix $x_1 A_1 + \cdots + x_d A_d$ are real,
we conclude that $q_x \in \RR[t]$ has only real zeroes.

A polynomial $p \in \RR[x_1,\ldots,x_d]$ is said to satisfy 
the {\em real zero} ($RZ$) condition with respect to $x^0 \in \RR^d$,
or to be a {\em $RZ_{x^0}$ polynomial},
if for all $x \in \RR^d$ the univariate polynomial $p_x(t) = p(x^0 + tx)$ has only real zeroes. 
It is clear that a divisor of a $RZ_{x^0}$ polynomial is again a $RZ_{x^0}$ polynomial.
We have thus arrived at the following result of \cite{HV07}.

\begin{thm} \label{thm:nc}
If a convex set $\cC$ with $x^0 \in \Int \cC$ admits
a LMI representation, then $\cC$ is an algebraic interior whose minimal defining polynomial $p$
is a $RZ_{x^0}$ polynomial. \eqref{eq:lmirep} is a LMI representation of $\cC$
(that is positive definite on $\Int \cC$) if and only if $A_0 + x^0_1 A_1 + \cdots + x^0_d A_d > 0$ and
$$
\det(A_0 + x_1 A_1 + \cdots + x_d A_d) = p(x) h(x),
$$
where $h \in \RR[x_1,\ldots,x_d]$ satisfies $h>0$ on $\Int \cC$.
\end{thm}

\subsection{}

The definition of a $RZ_{x^0}$ polynomial has a simple geometric meaning (\cite[Section 3]{HV07}).
Assume for simplicity that $p$ is reduced (i.e., without multiple irreducible factors) of degree $m$.
Then $p$ is a $RZ_{x^0}$ polynomial if and only if a general straight line through $x^0$ in $\RR^d$ intersects
the corresponding real affine hypersurface $\cV_p(\RR)$ (see \eqref{eq:real_affine_hyper})
in $m$ distinct points.
Alternatively, every straight line through $x^0$ in the real projective space $\PP^d(\RR)$
intersects the projective closure $\cV_P(\RR)$ of $\cV_p(\RR)$,
\begin{equation} \label{eq:real_proj_hyper}
\cV_P(\RR) = \{[X] \in \RR^d \colon P(X) = 0\},
\end{equation} 
in exactly $m$ points counting multiplicities.
Here we identify as usual the $d$ dimensional real projective
space $\PP^d(\RR)$ with
the union of $\RR^d$ and of the hyperplane
at infinity $X_0=0$,
so that the affine coordinates $x=(x_1,\ldots,x_d)$ and the
projective coordinates $X=(X_0,X_1,\ldots,X_d)$ are related by
$x_1=X_1/X_0$, \dots, $x_d=X_d/X_0$;
we denote by $[X] \in \PP^d(\RR)$ the point with the projective coordinates $X$;
and we let $P \in \RR[X_0,X_1,\ldots,X_d]$ be the homogenization of $p$,
\begin{equation} \label{eq:project}
P(X_0,X_1,\ldots,X_d) = X_0^m p(X_1/X_0,\ldots,X_d/X_0).
\end{equation} 
Notice that if $X=(1,x)$ and $X^0=(1,x^0)$,
\begin{equation} \label{eq:projvsaffine_onaline}
P(X + sX^0) = (s+1)^m p(x_0 + (s+1)^{-1}(x-x^0)).
\end{equation} 

It turns out that if $p$ is a $RZ_{x^0}$ polynomial with $p(x^0)>0$,
and if $x'$ belongs to the interior of the closure of the connected component of $x^0$ in 
$\{x \in {\mathbb R}^d \colon p(x) > 0\}$,
then $p(x')>0$ and $p$ is also a $RZ_{x'}$ polynomial (\cite[Section 5.3]{HV07}). 
We call an algebraic interior $\cC$ whose minimal defining polynomial
satisfies the $RZ$ condition with respect to one and then every point of $\Int \cC$
a {\em rigidly convex} algebraic interior. 

As simple examples, we see that the circle 
$\{(x_1,x_2) \colon x_1^2 + x_2^2 \leq 1\}$ is a rigidly convex algebraic interior,
while the ``flat TV screen''
$\{(x_1,x_2) \colon x_1^4 + x_2^4 \leq 1\}$ is not.
{\em Theorem \ref{thm:nc} tells us that a necessary condition for $\cC$ to admit a LMI representation
is that $\cC$ is a rigidly convex algebraic interior, and the size $n$ of the matrices in a LMI representation
is greater than or equal to the degree $m$ of a minimal defining polynomial $p$ of $\cC$.}

Rigidly convex algebraic interiors are always convex sets (\cite[Section 5.3]{HV07}).
They are also basic closed semialgebraic sets, as follows (\cite[Remark 2.6]{NPS10} following \cite{Re06}).
Let $p$ be a minimal defining polynomial of a rigidly convex algebraic interior $\cC$, of degree $m$, 
and let $x^0 \in \Int \cC$.
We set
\begin{equation} \label{eq:renegar_der}
P^{(k)}_{x^0}(X) = \frac{d^k}{ds^k} \left. P(X + s X^0) \right|_{s=0},\quad
p^{(k)}_{x^0}(x) = P^{(k)}_{x^0}(1,x_1,\ldots,x_d),
\end{equation}
where $P$ is the homogenization of $p$ (see \eqref{eq:project}) and $X^0=(1,x^0)$;
$p^{(k)}_{x^0}$ is called the {\em $k$th Renegar derivative} of $p$ with respect to $x^0$.
Then $p^{(k)}_{x^0}$ is a $RZ_{x^0}$ polynomial with $p^{(k)}_{x^0}(x^0)>0$ for all $k=1,\ldots,m-1$.
The rigidly convex algebraic interiors $\cC^{(k)}$ containing $x^0$ with minimal defining polynomials $p^{(k)}_{x^0}$
(i.e., the closures of the connected components of $x^0$ in 
$\{x \in {\mathbb R}^d \colon p^{(k)}_{x^0}(x) > 0\}$) are increasing:
$\cC = \cC^{(0)} \subseteq \cC^{(1)} \subseteq \cdots \subseteq \cC^{(m-1)}$,
and
\begin{equation} \label{eq:basic_descr}
\cC = \{x \in \RR^d \colon p(x) \geq 0, p^{(1)}_{x^0}(x) \geq 0, \ldots, p^{(m-1)}_{x^0}(x) \geq 0\}.
\end{equation}

$RZ$ polynomials can be also characterized by a very simple global topology of the corresponding real projective
hypersurface $\cV_P(\RR)$ (see \eqref{eq:real_proj_hyper}; readers who prefer can assume
that the corresponding real affine hypersurface $\cV_p(\RR)$ is compact in $\RR^d$ --- this implies that the degree $m$
of $p$ is even --- and replace in the following the real projective space $\PP^d(\RR)$ by the affine space $\RR^d$).
We call $W \subseteq \PP^d(\RR)$ an {\em ovaloid} if
$W$ is isotopic in $\PP^d(\RR)$
to a sphere $S \subset {\mathbb R}^m \subset {\mathbb P}^m({\mathbb R})$,
i.e., there is a homeomorphism $F$ of $\PP^d(\RR)$
with $F(S) = W$, and furthermore $F$ is homotopic to
the identity, i.e., there is a homeomorphism $H$ of
$[0,1] \times \PP^d(\RR)$ such that
$H_t=H|_{\{t\} \times \PP^d(\RR)}$ is a homeomorphism
of $\PP^d(\RR)$ for every $t$,
$H_0 = \operatorname{Id}_{\PP^d(\RR)}$,
and $H_1 = F$.
Notice that $\PP^d(\RR) \setminus S$ consists of two 
connected components
only one of which is contractible, hence the same is true
of $\PP^d(\RR) \setminus W$; we call the contractible component
the {\em interior} of the ovaloid $W$, and the non-contractible 
component the {\em exterior}. We call $W \subseteq \PP^d(\RR)$ a {\em pseudo-hyperplane}
if $W$ is isotopic in $\PP^d(\RR)$ to a (projective) hyperplane
$H \subseteq \PP^d(\RR)$. 
In the case $d=2$ we say {\em oval} and {\em pseudo-line} instead of ovaloid and pseudo-hyperplane.
We then have the following result; we refer to \cite[Sections 5 and 7]{HV07} for proof, discussion, and implications.

\begin{prop}
\label{prop:oval}
Let $p \in \RR[x_1,\ldots,x_d]$ be reduced of degree $m$ and assume that
the corresponding real projective hypersurface $\cV_P(\RR)$ is smooth.
Then $p$ satisfies $RZ_{x^0}$ with $p(x^0) \neq 0$ if and only if
\begin{itemize}
\item[a.]
if $m=2k$ is even, $\cV_P(\RR)$ is a disjoint union of $k$ ovaloids $W_1,\ldots,W_k$,
with $W_i$ contained in the interior of $W_{i+1}$, $i=1,\ldots,k-1$,
and $x^0$ lying in the interior of $W_1$;
\item[b.]
if $m=2k+1$ is odd, $\cV_P(\RR)$ is a disjoint union of $k$ ovaloids $W_1,\ldots,W_k$,
with $W_i$ contained in the interior of $W_{i+1}$, $i=1,\ldots,k-1$,
and $x^0$ lying in the interior of $W_1$, and a pseudo-hyperplane $W_{k+1}$
contained in the exterior of $W_k$.
\end{itemize}
\end{prop}

Let us denote by $\cI$ the interior of $W_1$, let us normalize $p$ by $p(x^0)>0$,
and let $H_\infty = \{X_0 = 0\}$ be the hyperplane at infinity in $\PP^d(\RR)$.
If $\cI \cap H_\infty = \emptyset$, then the closure of $\cI$ in $\RR^d$ is a rigidly convex algebraic
interior with a minimal defining polynomial $p$.
If $\cI \cap H_\infty \neq \emptyset$, then $\cI \setminus \cI \cap H_\infty$ consists of two connected
components, the closure of each one of them in $\RR^d$ being a rigidly convex algebraic
interior with a minimal defining polynomial $p$ (if $m$ is even) or 
$p$ for one component and $-p$ for the other component
(if $m$ is odd).

\section{Determinantal representations of polynomials: some of the known and of the unknown}
\label{sec:resultsdetrep}

\subsection{}

The following is proved in \cite[Section 5]{HV07} 
(based on the results of \cite{Vin93} and \cite{BV99}, see also \cite{Dub83}).

\begin{thm} \label{thm:2d}
Let $p \in \RR[x_1,x_2]$ be a $RZ_{x^0}$ polynomial of degree $m$ with $p(x^0)=1$.
Then there exist $A_0,A_1,A_2 \in \sym\mat{\RR}{m}$ with $A_0+x^0_1A_1+x^0_2A_2 = I$ such that
\begin{equation} \label{eq:2d_affine_detrep}
\det(A_0+x_1A_1+x_2A_2) = p(x).
\end{equation}
\end{thm}

We will review the proof of Theorem \ref{thm:2d} given in \cite{HV07} in Section \ref{sec:sheaves} below,
and then present in Section \ref{sec:interlace} an alternate proof for positive self-adjoint (rather than real
symmetric) determinantal representations that avoids
the transcendental machinery of Jacobian varieties and theta functions
(though it still involves, to a certain extent, meromorphic differentials on a compact Riemann surface).

{\em Theorem \ref{thm:2d} tells us that a necessary and sufficient condition for $\cC \subseteq \RR^2$ 
to admit a LMI representation
is that $\cC$ is a rigidly convex algebraic interior, and the size of the matrices in a LMI representation
can be taken equal to be the degree $m$ of a minimal defining polynomial $p$ of $\cC$.}

There can be no exact analogue of Theorem \ref{thm:2d} for $d>2$. 
Indeed, we have

\begin{prop} \label{prop:gen} 
A general polynomial $p \in \CC[x_1,\ldots,x_d]$ of degree $m$
does not admit a determinantal representation 
\begin{equation} \label{eq:affine_detrep_no_h}
\det(A_0+x_1A_1+\cdots+x_dA_d) = p(x),
\end{equation}
with $A_0,A_1,\ldots,A_d \in \mat{\CC}{m}$, for $d>3$ and for $d=3$, $m \geq 4$.
\end{prop}

Since for any fixed $x^0 \in \RR^d$ the set of $RZ_{x^0}$ polynomials of degree $m$ with $p(x^0)>0$
such that the corresponding real projective hypersurface $\cV_P(\RR)$ is smooth
is an open subset of the vector space of polynomials over $\RR$ of degree $m$
(see \cite[Sections 5 and 7]{HV07} following \cite{Nui68}),
it follows that a general $RZ_{x^0}$ polynomial $p \in \RR[x_1,\ldots,x_d]$ of degree $m$ with $p(x^0)>0$
does not admit a determinantal representation \eqref{eq:affine_detrep_no_h} with $m \times m$ matrices ---
even without requiring
real symmetry or positivity --- for $d>3$ and for $d=3$, $m \geq 4$.
(For the remaining cases when $d=3$, the case $m=2$ is straightforward and the case
$m=3$ is treated in details in \cite{BK07} when
the corresponding complex projective cubic surface $\cV_P$ in $\PP^3(\CC)$ is smooth; 
in both cases there are no positive real symmetric determinantal representations of size $m$
as in Theorem \ref{thm:2d}, but there are positive self-adjoint determinantal representations of size $m$,
i.e., representations \eqref{eq:affine_detrep_no_h} with $m \times m$ self-adjoint matrices such that
$A_0 + x^0_1A_1 + x^0_2A_2 + x^0_3A_3 = I$.)

Proposition \ref{prop:gen} follows by a simple count of parameters, see \cite{Dic21}.
It also follows from Theorem \ref{thm:kv} below using the Noether--Lefschetz theory
\cite{Lef24,Har70,GH85}, since for a general homogeneous polynomial $P \in \CC[X_0,X_1,\ldots,X_d]$ of degree $m$
with $d>3$ or with $d=3$, $m \geq 4$, the only line bundles on $\cV_P$ are of the form $\cO_{\cV_P}(j)$
and these obviously fail the conditions of the theorem.

The following is therefore the ``best possible'' generalization of Theorem \ref{thm:2d} to the case $d>2$.

\begin{conj} \label{conj:gend}
Let $p \in \RR[x_1,\ldots,x_d]$ be a $RZ_{x^0}$ polynomial of degree $m$ with $p(x^0)=1$.
Then there exists a $RZ_{x^0}$ polynomial $h \in \RR[x_1,\ldots,x_d]$ of degree $\ell$ with $h(x^0)=1$
and with the closure of the connected component of $x^0$ in 
$\{x \in {\mathbb R}^d \colon h(x) > 0\}$ containing the closure of the connected component of $x^0$ in 
$\{x \in {\mathbb R}^d \colon p(x) > 0\}$, 
and $A_0,A_1,\ldots,A_d \in \sym\mat{\RR}{n}$, $n \geq m+\ell$, 
with $A_0+x^0_1A_1+\cdots+x^0_dA_d = I$, such that
\begin{equation} \label{eq:affine_detrep}
\det(A_0+x_1A_1+\cdots+x_dA_d) = p(x) h(x).
\end{equation}
\end{conj}

Notice that is enough to require that $h$ is a polynomial that is strictly positive
on the connected component of $x^0$ in 
$\{x \in {\mathbb R}^d \colon h(x) > 0\}$, since it then follows from \eqref{eq:affine_detrep}
that $h$ is a $RZ_{x^0}$ polynomial with $h(x^0)=1$
and with the closure of the connected component of $x^0$ in 
$\{x \in {\mathbb R}^d \colon h(x) > 0\}$ containing the closure of the connected component of $x^0$ in 
$\{x \in {\mathbb R}^d \colon p(x) > 0\}$.

{\em Conjecture \ref{conj:gend} tells us that a necessary and sufficient condition for $\cC \subseteq \RR^d$ 
to admit a LMI representation
is that $\cC$ is a rigidly convex algebraic interior.}

We can also homogenize \eqref{eq:affine_detrep},
\begin{equation} \label{proj_detrep}
\det(X_0A_0+X_1A_1+\cdots+X_dA_d) = P(X) \tilde H(X),
\end{equation}
where
$\tilde H(X) = H(X) X_0^{n - m - \ell}$
and $P$ and $H$ are the homogenizations of $P$ and $H$ respectively (see \eqref{eq:project}).

\subsection{}

The easiest way to establish Conjecture \ref{conj:gend} would be to try taking $h=1$ in \eqref{eq:affine_detrep}
bringing us back to \eqref{eq:affine_detrep_no_h};
in the homogeneous version, 
$\tilde H = X_0^{n-m}$ in \eqref{proj_detrep}. This was the the form of the conjecture stated in \cite{HV07}.
It was given further credence by the existence of real symmetric determinantal representations
without the requirement of positivity.

\begin{thm} \label{thm:symdetrep}
Let $p \in \RR[x_1,\ldots,x_d]$. Then there exist 
$A_0,A_1,\ldots,A_d \in \sym\mat{\RR}{n}$ for some $n \geq m$ such that
$p$ admits the determinantal representation \eqref{eq:affine_detrep_no_h}.
\end{thm}

Theorem \ref{thm:symdetrep} was first established in \cite{HMcCV06} using free noncommutative techniques.
More precisely, the method was to take a lifting of $p$ to the free algebra and to apply
results of noncommutative realization theory to first produce a determinantal representation
with $A_0,A_1,\ldots,A_d \in \mat{\RR}{n}$ and then to show that it is symmetrizable;
see \cite[Section 14]{HMcCV06} for details and references.
An alternate proof of Theorem \eqref{thm:symdetrep} that uses more elementary arguments was given 
in \cite{Qua1}. 
As it turns out, determinantal representations also appear naturally in algebraic complexity theory,
and a proof of Theorem \eqref{thm:symdetrep} from this perspective was given in \cite{GKKP11}.

Unfortunately, the analogue of Theorem \ref{thm:symdetrep} for positive real symmetric (or positive self-adjoint)
determinantal representations fails. Counterexamples were first established in \cite{Br1}, and subsequently in \cite{NT1}.
Indeed we have

\begin{prop} \label{prop:genposetc}
A general $RZ_{x^0}$ polynomial $p \in \RR[x_1,\ldots,x_d]$ of degree $m$ with $p(x^0)=1$
does not admit a determinantal representation \eqref{eq:affine_detrep_no_h},
where $A_0,A_1,\ldots,A_d \in \herm\mat{\CC}{n}$ for some $n \geq m$
with $A_0+x^0_1A_1+\cdots+x^0_dA_d=I$,
for any fixed $m \geq 4$ and $d$ large enough or for any fixed $d \geq 3$ and $m$ large enough. 
\end{prop}

Here ``large enough'' means that $d^3 m^2 < \binom{m+d}{m} - 1$. We refer to \cite[Section 3]{NT1} for details
and numerous examples of $RZ$ polynomials that do not admit a positive self-adjoint determinantal representation
as in Proposition \ref{prop:genposetc}. One simple example is
\begin{equation} \label{eq:badquadratic}
p = (x_1+1)^2 - x_2^2 - \cdots - x_d^2
\end{equation}
for $d \geq 5$ (for $d=4$ this polynomial admits a positive self-adjoint determinantal representation
but does not admit a positive real symmetric determinantal representation).
The proofs are based on the fact that a positive self-adjoint (or real symmetric) determinantal representation
of size $n$ always contains, after a unitary (or orthogonal) transformation of the matrices $A_0,A_1,\ldots,A_d$,
a direct summand $I_{n-n'} + x_10_{n-n'}+\cdots+x_d0_{n-n'}$ ---
yielding a determinantal representation of size $n'$ ---
for relatively small $n'$:
one can always take $n' \leq md$ and in many instances one can actually take $n' = m$,
see \cite[Theorems 2.4 and 2.7]{NT1}.
It would be interesting to compare these results with the various general conditions for decomposability
of determinantal representations obtained in \cite{KV2,KV1}.

\subsection{}

The next easiest way to establish Conjecture \ref{conj:gend} is to try taking $h$ in \eqref{eq:affine_detrep}
to be a power of $p$, $h=p^{r-1}$,
so that we are looking for a positive real symmetric determinantal representation of $p^r$,
\begin{equation} \label{eq:affine_detrep_power}
\det(A_0 + x_1A_1 + \cdots + x_dA_d) = p(x)^r;
\end{equation}
in the homogeneous version, $\tilde H = P^{r-1} \cdot X_0^{n-mr}$ in \eqref{proj_detrep}.
If we do not require positivity or real symmetry,
then at least for $p$ irreducible,
$p^r$ admits a determinantal representation
\eqref{eq:affine_detrep_power} with 
$A_0,A_1,\ldots,A_d \in \mat{\CC}{n}$, $n=mr$, for some $r \in \NN$;
this follows by the theory of matrix factorizations \cite{Eis80}, 
since $pI_r$ can be written as a product of matrices with linear entries,
see \cite{BHS88,HUB91} (and also the references in \cite{NT1}).

As established in \cite{Br1}, the answer for positive real symmetric determinantal representations is again no.
Namely, let $p$ be a polynomial of degree $4$ in $8$ variables
labeled $x_a,x_b,x_c,x_d,x_{a'},x_{b'},x_{c'},x_{d'}$,
defined by
\begin{equation} \label{eq:vamos_poly}
p = \sum_{S \in \cB(V_8)} \prod_{j \in S} (x_j+1),
\end{equation}
where $\cB(V_8)$ is the set consisting of all $4$-element subsets
of $\{a,b,c,d,a',b',c',d'\}$ except for
$$
\{a,a',b,b'\},\ \{b,b',c,c'\},\ \{c,c',d,d'\},\ \{d,d',a,a'\},\ \{a,a',c,c'\}.
$$
$\cB(V_8)$ is the set of bases of a certain matroid $V_8$ 
on the set $\{a,b,c,d,a',b',c',d'\}$ called the {\em Vamos cube}.
Then 

\begin{thm} \label{thm:badvamos}
$p$ is $RZ$ with respect to $0$, and
for all $r \in \NN$, the polynomial $p^r$ does not admit a determinantal
representation \eqref{eq:affine_detrep_power}
where $A_0,A_1,\ldots,A_d \in \sym\mat{\RR}{n}$ for some $n \geq mr$
with $A_0=I$.
\end{thm}

This follows since on the one hand,
$V_8$ is a half-plane property matroid, and on the other hand,
it is not representable over any field,
more precisely its rank function does not satisfy Ingleton inequalities.
See \cite[Section 3]{Br1} for details.
Notice that it turns out that one can take without loss of generality $n=mr$ in Theorem \ref{thm:badvamos}, 
see the paragraph following Proposition \ref{prop:genposetc} above.
Notice also that because of the footnote on page \pageref{foot:herm}, it does not matter here
whether we are considering real symmetric or self-adjoint determinantal representations.

The polynomial \eqref{eq:vamos_poly} remains so far the only example of a $RZ$ polynomial no power of which
admits a positive real symmetric determinantal representation\footnote{
Peter Br\" and\' en noticed
(see \url{http://www-e.uni-magdeburg.de/ragc/talks/branden.pdf}) that one can use 
the symmetry of the polynomial (14) to produce from it a $RZ$ polynomial in $4$ variables
no power of which admits a positive real symmetric determinantal representation.}.
For instance, we have

\begin{thm} \label{thm:goodquadratic}
Let $p \in \RR[x_1,\ldots,x_d]$ be a $RZ_{x^0}$ polynomial of degree $2$ with $p(x^0)=1$.
Then there exists $r \in \NN$
and $A_0,A_1,\ldots,A_d \in \sym\mat{\RR}{n}$, $n=mr$, 
with $A_0+x^0_1A_1+\cdots+x^0_dA_d = I$, such that $p^r$ admits the determinantal representation
\eqref{eq:affine_detrep_power}.
\end{thm}

Theorem \ref{thm:goodquadratic} has been established in \cite{NT1} using Clifford algebra techniques.
More precisely, one associates to a polynomial $p \in \RR[x_1,\ldots,x_d]$ of degree $m$
a unital $^*$-algebra as follows.
Let $\CC\langle z_1,\ldots,z_d \rangle$ be the free $^*$-algebra on $d$ generators, i.e., 
$z_1,\ldots,z_d$ are noncommuting self-adjoint indeterminates. For the homogenization $P$ of $p$,
we can write
$$
P(-x_1z_1-\cdots-x_dz_d,x_1,\ldots,x_d) = \sum_{k \in \ZZ^d_+,\,|k|=m} q_k(z) x^k,
$$
for some $q_k \in \CC\langle z_1,\ldots,z_d \rangle$, where $k=(k_1,\ldots,k_d)$,
$|k|=k_1+\cdots+k_d$, and $x^k=x_1^{k_1} \cdots x_d^{k_d}$.
We define the {\em generalized Clifford algebra associated with $p$} to be the quotient
of $\CC\langle z_1,\ldots,z_d \rangle$ by the two-sided ideal generated by $\{q_k\}_{|k|=m}$.
It can then be shown that at least if $p$ is irreducible,
$p^r$ admits a self-adjoint determinantal representation
\eqref{eq:affine_detrep_power} of size $mr$ with $A_0=I$ for some $r \in \NN$ if and only if
the generalized Clifford algebra associated with $p$ admits a finite-dimensional
unital $^*$-representation. In case $m=2$ and $p$ is an irreducible $RZ_0$ polynomial,
the generalized Clifford algebra associated with $p$ turns out to be ``almost'' the usual Clifford algebra,
yielding the proof of Theorem \ref{thm:goodquadratic}. For details and references,
see \cite[Sections 4 and 5]{NT1}. 
It would be interesting to investigate the generalized Clifford algebra associated with 
the polynomial \eqref{eq:vamos_poly}.\footnote{
Tim Netzer recently reported
(see \url{http://www-e.uni-magdeburg.de/ragc/talks/netzer.pdf}) 
that for an irreducible $RZ_0$ polynomial $p$ with $p(0)=1$,
Conjecture 5 holds (with $x^0=0$) if and only if $-1$ is not a sum of hermitian squares
in the generalized Clifford algebra associated with $p$.}

A new obstruction to powers of $p$ admitting a positive real symmetric determinantal representation
has been recently discovered in \cite{NPT1}. It is closely related to the question
of how to test a polynomial for the $RZ$ condition, see \cite{Hen10}.
For any monic polynomial $f \in \RR[t]$ of degree $m$ with zeroes $\lambda_1,\ldots,\lambda_m$,
let us define the matrix $H(f) = \left[h_{ij}\right]_{i,j=1,\ldots,m}$
by $h_{ij} = \sum_{k=1}^d \lambda_k^{i+j-2}$; notice that $h_{ij}$ are actually polynomials
in the coefficients of $f$. $H(f)$ is called the Hermite matrix of $f$,
and it is positive semidefinite if and only if all the zeroes of $f$ are real.
Given $p \in \RR[x_1,\ldots,x_d]$ of degree $m$ with $p(x^0)=1$, 
we now consider $H(\check p_x)$ where $\check p_x(t)=t^m p(x^0+t^{-1}x)$;
it is a polynomial matrix that we call the {\em Hermite matrix} of $p$ with respect to $x^0$ and denote
$H(p;x^0)$.
$p$ is a $RZ_{x^0}$ polynomial if and only if $H(p;x^0)(x) \geq 0$ for all $x \in \RR^d$.
Now, it turns out that if there exists $r \in \NN$ such that
$p^r$ admits a determinantal representation \eqref{eq:affine_detrep_power}
with $A_0,A_1,\ldots,A_d \in \sym\mat{\RR}{n}$, $n=mr$, 
and $A_0+x^0_1A_1+\cdots+x^0_dA_d = I$, then $H(p;x^0)$ can be factored:
$H(p;x^0) = Q^\trans Q$ for some polynomial matrix $Q$, i.e., $H(p;x^0)$ is a {\em sum of squares}.
We notice that $H(p;x^0)$ can be reduced by homogeneity to a polynomial matrix in $d-1$ variables,
implying that the sum of squares decomposition (factorization) is not an obstruction in the case $d=2$,
but it is in the case $d>2$. In particular, there is numerical evidence that for the polynomial $p$
of \eqref{eq:vamos_poly} the Hermite matrix (with respect to $0$) is not a sum of squares.
We refer to \cite{NPT1} for details.
It would be very interesting to use these ideas in the case $d=2$
to obtain a new proof of a weakened version of Theorem \ref{thm:2d} that gives a positive real symmetric
determinantal representation of $p^r$ (of size $mr$) for some $r \in \NN$.

\subsection{}

There have been so far no attempts to pursue Conjecture \ref{conj:gend}
with other choices of $h$ than $1$ or a power of $p$.
Two natural candidates are products of (not necessarily distinct) linear forms
(that are nonnegative on the closure of the connected component of $x^0$ in 
$\{x \in {\mathbb R}^d \colon p(x) > 0\}$),
and products of powers of Renegar derivatives of $p$ with respect to $x^0$ (see \eqref{eq:renegar_der}).

Conjecture \ref{conj:gend} is a reasonable generalization of Theorem \ref{thm:2d} for the purposes of LMI representations
of convex sets (provided the solution gives a good hold of the extra factor $h$ and of the size $n$).
It is less satisfactory as a means of describing or generating $RZ$ polynomials. The following alternative conjecture,
that was proposed informally by L.~Gurvits, might be more useful for that purpose. It is based on the fact that we have 
two systematic ways of generating $RZ$ polynomials: positive real symmetric (or self-adjoint) 
determinantal representations and Renegar derivatives.

\begin{conj} \label{conj:gurvits}
Let $p \in \RR[x_1,\ldots,x_d]$ be a $RZ_{x^0}$ polynomial of degree $m$ with $p(x^0)=1$.
Then there exist $k \in \ZZ_+$, a $RZ_{x^0}$ polynomial $q \in \RR[x_1,\ldots,x_d]$ of degree $m+k$
such that $p=q^{(k)}_{x^0}$, and
$A_0,A_1,\ldots,A_d \in \sym\mat{\RR}{(m+k)}$ 
with $A_0+x^0_1A_1+\cdots+x^0_dA_d = I$, such that
$$
\det(A_0+x_1A_1+\cdots+x_dA_d) = q(x).
$$
\end{conj}
 
\section{Determinantal representations of homogeneous polynomials and sheaves on projective hypersurfaces}
\label{sec:sheaves}

The kernel of a determinantal representation of a homogeneous polynomial is a sheaf 
on the corresponding projective hypersurface from which the representation itself can be reconstructed.
We consider here the ways to do so that use the duality between the kernel and the left kernel;
this gives the only known approaches to the proof of Theorem \ref{thm:2d}.
For a different way using the resolution of the kernel sheaf see \cite{Bea00}; we refer also to the bibliography
in \cite{Bea00,KV1} and to \cite[Chapter 4]{Dol1} and the references therein
for more about this old topic in algebraic geometry.

\subsection{} \label{subsec:detrep_and_adj}

Let $P \in \CC[X_0,X_1,\ldots,X_d]$ ($d > 1$) be a reduced (i.e., without multiple irreducible factors)
homogeneous polynomial of degree $m$, and let
\begin{equation} \label{eq:proj_hyper}
\cV_P = \{[X] \in \CC^d \colon P(X) = 0\}
\end{equation} 
be the corresponding complex projective hypersurface. Notice that 
when $P$ is a polynomial over $\RR$, $\cV_P$ is naturally endowed with an antiholomorphic
involution $\tau$ (the complex conjugation or the Galois action of $\text{Gal}(\CC/\RR)$)
and the set of fixed points of $\tau$ is exactly 
the real projective hypersurface $\cV_P(\RR)$ as in \eqref{eq:real_proj_hyper}.
Let
\begin{multline} \label{proj_detrep_P}
\det(X_0A_0+X_1A_1+\cdots+X_dA_d) = P(X),\\
A_\alpha=\left[A_{\alpha,ij}\right]_{i,j=1,\ldots,m} \in \mat{\CC}{m},\, 
\alpha=0,1,\ldots,d,
\end{multline}
be a determinantal representation of $P$, and let
\begin{equation} \label{eq:U_and_V}
U = X_0A_0+X_1A_1+\cdots+X_dA_d,\quad V = \left[V_{ij}\right]_{i,j=1,\ldots,m} = \adj U,
\end{equation}
where $\adj Y$ denotes the adjoint matrix of a $m \times m$ matrix $Y$, i.e., 
the matrix whose $(i,j)$ entry is $(-1)^{i+j}$ times the determinant of the matrix obtained from $Y$
by removing the $j$th row and the $i$th column, so that $Y \cdot \adj Y = \det Y \cdot I$.
Notice that
\begin{align}
\label{eq:detadj}
\det V &= P^{m-1},\\
\label{eq:adjadj}
\adj V &= P^{m-2} \cdot U.
\end{align}
Notice also that using the formula for the differentiation of a determinant and row expansion,
\begin{equation} \label{eq:derdetrep}
\frac{\partial P}{\partial X_\alpha} = \sum_{l,k=1}^n A_{\alpha,lk} V_{kl}.
\end{equation}
In particular, $V(X)$ is not zero for a smooth point $[X]$ of the hypersurface $\cV_P$,
so that $V(X)$ has rank $1$ there and $U(X)$ has rank $m-1$.

\subsection{} \label{subsec:sheaves}

We restrict our attention now to the case $d=2$, i.e., $\cV_P$ is a projective plane curve.
Let us assume for a starter
that $P$ is irreducible and that $\cV_P$ is smooth ---
we will explain how to relax this assumption in Section \ref{subsec:sing} below. 
Then we conclude that
$\cL([X]) = \ker U(X)$ is a one-dimensional subspace of $\CC^m$ for all points $[X]$ on $\cV_P$, and
these subspaces glue together to form a line bundle $\cL$ on $\cV_P$;
more precisely, $\cL$ is a subbundle of the trivial rank $m$ vector bundle
$\cV_P \times \CC^m$ whose fiber at the point $[X]$
equals $\cL([X])$. It is convenient to twist and define $\cE = \cL(m-1)$.
More algebraically, $\cE$ is determined by the exact
sequence of sheaves on $\cV_P$
\begin{equation} \label{eq:exact_seq}
0 \longrightarrow \cE \longrightarrow \cO_{\cV_P}^{\oplus m}(m-1) \overset{U}{\longrightarrow}
\cO_{\cV_P}^{\oplus m}(m) \longrightarrow \coker(U) \longrightarrow 0,
\end{equation}   
where $U$ denotes the operator of right multiplication by the matrix acting on columns.
The following are some of the properties of the kernel line bundle.
\begin{enumerate}
\item
The determinantal representation is determined up to a natural equivalence 
(multiplication on the left and on the right by constant invertible matrices)
by the isomorphism class of the line bundle $\cE$.
\item
The columns $F_j = \left[V_{ij}\right]_{i=1,\ldots,m}$ of the adjoint matrix $V$ form a basis for the space
$H^0(\cE,\cV_P)$ of global sections of $\cE$.
\item
$\cE$ satisfies $h^0(\cE(-1))=h^1(\cE(-1))=0$.
\end{enumerate}
See \cite{CT79,Vin89,Bea00} for details. By the Riemann--Roch theorem, $\cE(-1)$ is a line bundle of degree $g-1$
on $\cV_P$ (where $g$ denotes the genus), and it is general in that it has no global sections, i.e., 
it lies on the complement
of the theta divisor in the Jacobian of $\cV_P$.

There is a similarly defined line bundle $\cL_\ell$ on $\cV_P$
with fibres $\cL_\ell([X]) = \ker_\ell U(X)$,
where $\ker_\ell$ denotes the left kernel of a matrix (a subspace of $\rmat{\CC}{1}{m}$);
we set, analogously, $\cE_\ell = \cL_\ell(m-1)$.
$\cE_\ell$ is defined by an exact sequence similar to \eqref{eq:exact_seq}
except that 
$U$ is now acting as the operator of left multiplication by the matrix on rows.
The rows $G_i = \left[V_{ij}\right]_{j=1,\ldots,m}$ of the adjoint matrix $V$ form a basis for the space
$H^0(\cE_\ell,\cV_P)$ of global sections of $\cE_\ell$.
There is furthermore a nondegenerate pairing
$\cE \times \cE_\ell \to \cK_{\cV_P}(2)$ 
(here $\cK_{\cV_P} \cong \cO_{\cV_P}(m-3)$ is the canonical line bundle on $\cV_P$), i.e.,
$\cE_\ell(-1)$ is isomorphic to the Serre dual $(\cE(-1))^* \otimes \cK_{\cV_P}$ of $\cE(-1)$, 
which is key to the reconstruction of the determinantal representation
from the corresponding line bundle.

Notice that if $P$ is a polynomial over $\RR$ and the determinantal representation is self-adjoint then 
$\cE_\ell \cong \cE^\tau$, whereas
if the determinantal representation is real symmetric, then $\cE_\ell \cong \cE^\tau \cong \cE$.
(In fact, in the real symmetric case the line bundle $\cE$ is defined over $\RR$ which is a somewhat
stronger condition than $\cE^\tau \cong \cE$ but the two actually coincide if $\cV_P(\RR) \neq \emptyset$,
see \cite{Vin93} and the references there.)

\subsection{} \label{subsec:theta}

There are two ways to define the pairing $\cE \times \cE_\ell \to \cK_{\cV_P}(2)$. One way, originating in multivariable
operator theory and multidimensional system theory, simply pairs the right and left kernels of the matrix
$U(X)$ against appropriate linear combinations of the coefficient matrices $A_0$, $A_1$, $A_2$;
see \cite{BV96}. This leads to explicit formulae for the coefficient matrices in terms of theta functions,
given a line bundle $\cE(-1)$ on $\cV_P$ with $h^0(\cE(-1))=h^1(\cE(-1))=0$, see \cite[Theorems 4.1 and 5.1]{BV99}. 
It is obvious from these formulae
that choosing $\cE(-1)$ with $(\cE(-1))^* \otimes \cK_{\cV_P} \cong \cE(-1)^\tau \cong \cE(-1)$ 
(i.e., $\cE(-1)$ is a real theta characteristic on $\cV_P$)
yields a real symmetric determinantal representation
(at least in the case $\cV_P(\RR) \neq \emptyset$). 
\cite[Section 4]{HV07} verifies (using the tools developed in \cite{Vin93})
that in case the dehomogenization $p(x_1,\ldots,x_d)=P(1,x_1,\ldots,x_d)$ of the original polynomial $P$
is $RZ$, appropriate choices of $\cE(-1)$ will yield a positive determinantal representation
(to be more precise, the positivity is ``built in'' \cite[(4.1)--(4.3)]{HV07}). 
``Appropriate choices'' means that the line bundle $\cE(-1)$ of degree $g-1$ 
(more precisely, its image under the Abel--Jacobi map) has to belong to 
a certain distinguished real $g$-dimensional torus $T_0$ in the Jacobian of $\cV_P$,
see \cite[Sections 3 and 4]{Vin93}; accidentally, this already forces
$\cE(-1)$ to be in the complement of the theta divisor, i.e., 
the condition $h^0(\cE(-1))=h^1(\cE(-1))=0$ becomes automatic.
It is interesting to notice that recent computational advances in theta functions on Riemann surfaces
make this approach possibly suitable for computational purposes, see \cite{PSV2}.

\subsection{} \label{subsec:adjmatrix}

Another way to define the pairing $\cE \times \cE_\ell \to \cK_{\cV_P}(2)$ is more algebraic and goes back to the classical
paper \cite{D1900}; it uses the adjoint matrix $V$ of the determinantal representation.
This leads to the following construction of the determinantal representation
given a line bundle $\cE(-1)$ on $\cV_P$ with $h^0(\cE(-1))=h^1(\cE(-1))=0$, see \cite{D1900,Vin89,BV96}.
Take bases $\{F_1,\ldots,F_m\}$ and $\{G_1,\ldots,G_m\}$ for the spaces of global sections
of $\cE$ and of $\cE_\ell$, respectively, where $\cE_\ell(-1) := (\cE(-1))^* \otimes \cK_{\cV_P}$ is the Serre dual.
Then $V_{ij}:=\langle F_j, G_i \rangle$ is a global section of $\cK_{\cV_P}(2) \cong \cO_{\cV_P}(m-1)$, 
hence a homogeneous polynomial
in $X_0,X_1,X_2$ of degree $m-1$. 
It can be shown that the matrix $V = \left[V_{ij}\right]_{i,j=1,\ldots,m}$
has rank $1$ on $\cV_P$, implying that \eqref{eq:detadj} holds, up to a constant factor $c$,
and that every entry of $\adj V$ is divisible by $P^{m-2}$. We can now define
a matrix $U$ of linear homogeneous forms by \eqref{eq:adjadj}, and it will be a determinantal representation
of $P$, up to the constant factor $c^{m-1}$.
It remains only to show that the constant factor is not zero, i.e., that 
$\det V$ is not identically zero.
This follows by choosing the bases for the spaces of global sections adapted to a straight line,
so that $V$ becomes diagonal along that line, and uses essentially the condition $h^0(\cE(-1))=h^1(\cE(-1))=0$.

It is quite straightforward that if $\cE$ satisfies 
$(\cE(-1))^* \otimes \cK_{\cV_P} \cong (\cE(-1))^\tau \cong \cE(-1)$ 
we obtain a real symmetric determinantal representation
(at least in the case $\cV_P(\RR)\neq\emptyset$, since we really need $\cE$ to be defined over $\RR$), 
whereas if $(\cE(-1))^* \otimes \cK_{\cV_P} \cong \cE(-1)^\tau$ we obtain a self-adjoint determinantal representation.

\subsection{} \label{subsec:divisors}

The above procedure can be written down more explicitly in terms of divisors and linear systems. 
We recall that for a homogeneous polynomial $F \in \CC[X_0,X_1,X_2]$,
the divisor $(F)$ of $F$ on $\cV_P$ is the formal sum of the zeroes of $F$ on $\cV_P$ with the orders
of the zeroes as coefficients 
(the order of the zero equals also the intersection multiplicity of the curves $\cV_Q$ and $\cV_P$ ---
here $Q$ can have multiple irreducible factors so that the curve $\cV_Q$ can have multiple components, i.e.,
it may be a non-reduced subscheme of $\PP^2$ over $\CC$).
 
Let $Q \in \CC[X_0,X_1,X_2]$ be an auxilliary homogeneous polynomial of degree $m-1$,
together with a decomposition 
$(Q) = D + D_\ell$, $\deg D = \deg D_\ell = m(m-1)/2$.
We assume that $D$ and $D_\ell$ satisfy the condition that $D-(L)$ or equivalently $D_\ell-(L)$ is not linearly equivalent
to an effective divisor on $\cV_P$, where $L$ is a linear form.

Take a basis $\{V_{11},\ldots,V_{m1}\}$ of the vector space of homogeneous polynomials of degree $m-1$
that vanish on $D$, with $V_{11}=Q$,
and a basis $\{V_{11},\ldots,V_{1m}\}$ of the vector space of homogeneous polynomials of degree $m-1$
that vanish on $D_\ell$. Write $(V_{i1}) = D + D_{\ell,i}$
and $(V_{1j}) = D_j + D_\ell$, where $D_1=D$ and $D_{\ell,1} = D_\ell$.
Define homogeneous polynomials $V_{ij}$ of degree $m-1$ for $i>1$ and $j>1$ by
$(V_{ij}) = D_j+D_{\ell,i}$. We then set $V=\left[V_{ij}\right]_{i,j=1,\ldots,m}$,
and obtain a determinantal representation $U$ of $P$ by \eqref{eq:adjadj}.

To be able to obtain a real symmetric determinantal representation of a polynomial $P$ over $\RR$,
we need $\cV_Q$ to be 
a real contact curve of $\cV_P$, i.e., to be defined by a polynomial $Q$ over $\RR$ and to have  
even intersection multiplicity at all points of intersection (in this case $D = D_\ell$ is uniquely determined).
To be able to obtain a self-adjoint determinantal representation we need $\cV_Q$ to be 
a real curve that is contact to $\cV_P$ at all real points of intersection
(in this case the real points of $D$ and of $D_\ell = D^\tau$ are uniquely determined whereas the non-real points
can be shuffled between the two).

\subsection{} \label{subsec:adjmatrix_positivity}

Unlike the approach of Section \ref{subsec:theta}, the approach of Sections \ref{subsec:adjmatrix}--\ref{subsec:divisors} 
does not produce directly the coefficient matrices of the determinantal representation,
so it is not clear a priori how to obtain a real symmetric or self-adjoint representation that is positive.
A delicate calculation with differentials carried out in \cite[Sections 4--6]{Vin93} shows that
that this will happen exactly in case the original polynomial $p$
is $RZ$ and $\cE(-1)$ (more precisely, its image under the Abel--Jacobi map) belongs to 
the distinguished real $g$-dimensional torus $T_0$ in the Jacobian of $\cV_P$.
We will obtain a corresponding result in terms of the auxilliary curve $\cV_Q$
in Section \ref{sec:interlace} below by elementary methods.

\subsection{} \label{subsec:sing}

We consider now how to relax the assumption that $\cV_P$ is irreducible and smooth.
A full analysis of determinantal representations for a general reduced polynomial $P$
involves torsion free sheaves of rank 1 on a possibly reducible and singular curve;
see \cite{KV1} and the references therein. However one can get far enough to obtain 
a full proof of Theorem \ref{thm:2d} by considering a restricted class of determinantal representations.

Let $\nu \colon \tcV_P \to \cV_P$ be the normalization or equivalently the desingularization. 
$\tcV_P$ is a disjoint union of smooth complex projective curves (or compact Riemann surfaces)
corresponding to the irreducible factors of $P$ (the irreducible components of $\cV_P$)
and 
$$
\nu\left|_{\tcV_P \setminus \nu^{-1}((\cV_P)_\sing)}\right. \colon 
\tcV_P \setminus \nu^{-1}((\cV_P)_\sing) \to \cV_P \setminus (\cV_P)_\sing
$$
is a (biregular or complex analytic) isomorphism, where $(\cV_P)_\sing$ denotes the set of singular points of $\cV_P$.
Let $\lambda \in (\cV_P)_\sing$; we assume that $\lambda$ lies in the affine plane $\CC^2 \subseteq \PP^2(\CC)$
(otherwise we just choose different affine coordinates near $\lambda$).
For every $\mu \in \nu^{-1}(\lambda)$ (i.e., for every branch of $\cV_P$ at $\lambda$),
the differential 
$$
\nu^* \left(\frac{dx_1}{\partial p/\partial x_2}\right) = - \nu^*\left(\frac{dx_2}{\partial p/\partial x_1}\right)
$$ 
on $\tcV_P$ has a pole at $\mu$; we denote the order of the pole by $m_\mu$.
We define
$$
\Delta_\lambda = \sum_{\mu \in \nu^{-1}(\lambda)} m_\mu \mu
$$
(the adjoint divisor of $\lambda$), and
\begin{equation} \label{eq:adjoint_div}
\Delta = \sum_{\lambda \in (\cV_P)_\sing} \Delta_\lambda
\end{equation}
(the adjoint divisor, or the divisor of singularities, of $\cV_P$); see, e.g., \cite[Appendix A2]{ACGH85}.

A determinantal representation $U$ of $P$ is called {\em fully saturated}
(or $\tcV_P/\cV_P$ saturated) if all the entries of the adjoint matrix $V$ vanish on 
the adjoint divisor: $(\nu^* V_{ij}) \geq \Delta$ for all $i,j=1,\ldots,m$.
This is a somewhat stronger condition than being a {\em maximal} (or {\em maximally generated})
determinantal representation, which means that for every $\lambda \in (\cV_P)_\sing$,
$\dim \ker U(\lambda)$ has the maximal possible dimension equal to the multiplicity of $\lambda$ on $\cV_P$.
We refer to \cite{KV1,KV2} for details. If $P$ is reducible than a fully saturated determinantal representation
always decomposes, up to equivalence, as a direct sum of determinantal representations of 
the irreducible factors of $P$;
hence we can assume that $P$ is irreducible.

For a fully saturated determinantal representation $U$ of $P$, we can define a line bundle $\tcL$ on 
$\tcV_P \setminus \nu^{-1}((\cV_P)_\sing)$ with fibres $\tcL([X]) = \ker U(X)$ and then extend it uniquely
to all of $\tcV_P$; we then define $\tcE = \tcL(m-1)(-\Delta)$, see \cite{BV96} --- here
$\tcL(m-1) = \tcL \otimes \nu^* \cO_{\cV_P}(m-1)$. Alternatively, we can define $\tcE = \nu^* \cE$,
where the sheaf $\cE$ on $\cV_P$ is still defined by \eqref{eq:exact_seq}, see \cite{KV1}.
We introduce similarly the left kernel line bundle $\tcE_\ell$.
Most of Sections \ref{subsec:sheaves}--\ref{subsec:adjmatrix} and \ref{subsec:adjmatrix_positivity} now carry over 
for a fully saturated determinantal representation $U$ of $P$ and line bundles $\tcE$ and $\tcE_\ell$
on $\tcV_P$; notice that the canonical line bundle on $\tcV_P$ is given by
$\cK_{\tcV_P} \cong \nu^* \cO_{\cV_P}(m-3)(-\Delta)$.
 
In Section \ref{subsec:divisors}, we have to take the auxilliary polynomial $Q$ 
to vanish on the adjoint divisor:
$(\nu^* Q) \geq \Delta$, with a decomposition $(\nu^* Q) = D + D_\ell + \Delta$.
We then take a basis $\{V_{11},\ldots,V_{m1}\}$ of the vector space of homogeneous polynomials of degree $m-1$
that vanish on $D$ and on the adjoint divisor, with $V_{11}=Q$,
and a basis $\{V_{11},\ldots,V_{1m}\}$ of the vector space of homogeneous polynomials of degree $m-1$
that vanish on $D_\ell$ and on the adjoint divisor; 
we write $(V_{i1}) = D + D_{\ell,i} + \Delta$
and $(V_{1j}) = D_j + D_\ell + \Delta$, where $D_1=D$ and $D_{\ell,1} = D_\ell$;
and we define homogeneous polynomials $V_{ij}$ of degree $m-1$ for $i>1$ and $j>1$ by
$(V_{ij}) = D_j+D_{\ell,i}+\Delta$.

\subsection{} \label{subsec:kerner}

The recent work \cite{KV1} extends the construction of the adjoint matrix of a determinantal representation
outlined in Section \ref{subsec:adjmatrix} to the most general higher dimensional situation.
Let $P=P_1^{r_1} \cdots P_k^{r_k} \in \CC[X_0,X_1,\ldots,X_d]$, 
where $P_1,\ldots,P_k$ are (distinct) irreducible polynomials,
and let 
\begin{equation} \label{eq:proj_hyper_scheme}
\cV_P = \Proj \CC[X_0,X_1,\ldots,X_d] / \langle P \rangle
\end{equation} 
be the corresponding closed subscheme of $\PP^n$ over $\CC$; of course $\cV_P$ is in general highly non-reduced. 
Let $U$ be a determinantal representation of $P$ as in \eqref{proj_detrep_P}--\eqref{eq:U_and_V};
we define the kernel sheaf $\cE$ on $\cV_P$ by the exact sequence \eqref{eq:exact_seq}, as before.
$\cE$ is a torsion-free sheaf on $\cV_P$ of multirank $(r_1,\ldots,r_k)$ 
(these notions have to be somewhat carefully defined), and
we have 
\begin{equation} \label{eq:vanishing}
h^0(\cE(-1))=h^{d-1}(\cE(1-d))=0,\quad h^i(\cE(j))=0,\,i=1,\ldots,d-2,\,j\in\ZZ.
\end{equation}
Conversely, 

\begin{thm} \label{thm:kv}
Let $\cE$ be a torsion-free sheaf on $\cV_P$ of multirank $(r_1,\ldots,r_k)$
satisfying the vanishing conditions \eqref{eq:vanishing}; 
then $\cE$ is the kernel sheaf of a determinantal representation
of $P$.
\end{thm}

As in Section \ref{subsec:adjmatrix}, Theorem \ref{thm:kv} is proved by taking bases of
$H^0(\cE,\cV_P)$ and of $H^0(\cE_\ell,\cV_P)$, $\cE_\ell = \cE^* \otimes \bomega_{\cV_P}(d)$
(here $\bomega_{\cV_P} = \cO_{\cV_P}(m-d-1)$ is the dualizing sheaf),
pairing these bases to construct a matrix $V$ of homogeneous polynomials of degree $m-1$,
and then defining the determinantal representation $U$ by \eqref{eq:adjadj};
there are quite a few technicalities, especially because the scheme is non-reduced. 
For $P$ a polynomial over $\RR$, the determinantal representation can be taken to be self-adjoint if (and only if)
$\cE^\tau \cong \cE^* \otimes \bomega_{\cV_P}(d)$ where $\tau$ is again the complex conjugation.
It should be also possible to characterize real symmetric determinantal representations.
(Complex symmetric determinantal representations correspond to $\cE \cong \cE^* \otimes \bomega_{\cV_P}(d)$.)

Theorem \ref{thm:kv} provides a new venue for pursuing Conjecture \ref{conj:gend}. To make it effective requires
progress in two directions:
\begin{enumerate}
\item
Given a reduced homogeneous polynomial $P$, characterize large classes of homogeneous polynomials $\tilde H$
such that the scheme $\cV_{P \tilde H}$ admits torsion free sheaves of correct multirank satisfying
the vanishing conditions \eqref{eq:vanishing}.
\item
If $p$ is $RZ$, characterize positive real symmetric or self-adjoint determinantal representations of $P$ in terms
of the kernel sheaf $\cE$. This is interesting not only for the general conjecture but also for special cases,
compare the recent paper \cite{DI1} dealing with
singular nodal quartic surfaces in $\PP^3$. It could be that the results of Section \ref{sec:interlace} below
admit some kind of a generalization.
\end{enumerate}

\section{Interlacing $RZ$ polynomials and positive self-adjoint determinantal representations}
\label{sec:interlace}

\subsection{} \label{subsec:def_interlace}

Let $p \in \RR[x_1,\ldots,x_d]$ be a reduced (i.e., without multiple factors) $RZ_{x^0}$ polynomial of degree $m$
with $p(x^0) \neq 0$,
and let $P$ be the homogenization of $p$ (see \eqref{eq:project}). 
Let $Q \in \RR[X_0,X_1,\ldots,X_d]$ be a homogeneous polynomial of degree $m-1$ that is relatively prime with $P$.
We say that $Q$ {\em interlaces} $P$ if for a general $X \in \RR^{d+1}$, there is a zero of
the univariate polynomial $Q(X + sX^0)$ in the open interval
between any two zeroes of the univariate polynomial $P(X + sX^0)$, where $X^0=(1,x^0)$.
Alternatively, for any $X \in \RR^{d+1}$, 
\begin{equation} \label{eq:alternate}
s_1 \leq s_1' \leq s_2 \leq \cdots \leq s_{m-1} \leq s_{m-1}' \leq s_m,
\end{equation}
where $s_1,\ldots,s_m$ are the zeroes of $P(X + sX^0)$ and $s_1',\ldots,s_{m-1}'$ are zeroes of $Q(X + sX^0)$,
counting multiplicities.
Notice (see \eqref{eq:projvsaffine_onaline}) that we can consider instead the zeroes of the univariate polynomials
$\check p_x(t) = t^m p(x_0 + t^{-1}x)$ and 
$\check \qut{}_x(t) = t^{m-1} \qut(x_0 + t^{-1}x)$ for a general or for any $x \in \RR^d$,
where $\qut(x)=Q(1,x_1,\ldots,x_d)$.
It follows that $\qut$ is a $RZ_{x^0}$ polynomial with $\qut(x^0) \neq 0$,
and (upon normalizing $p(x^0)>0$, $\qut(x^0)>0$)
the closure of the connected component of $x^0$ in 
$\{x \in {\mathbb R}^d \colon \qut(x) > 0\}$ contains the closure of the connected component of $x^0$ in 
$\{x \in {\mathbb R}^d \colon p(x) > 0\}$.
The degree of $\qut$ is either $m-1$ (in which case $Q$ is the homogenization of $\qut$)
or $m-2$ (in which case $Q$ is the homogenization of $\qut$ times $X_0$).

Geometrically, let $\yashar$ be a general straight line through $[X^0]$ in $\PP^d(\RR)$.
Then $Q$ interlaces $P$ if and only if any there is an intersection of $\yashar$ with the real projective hypersurface
$\cV_Q(\RR)$ in any open interval on $\yashar \setminus \{[X^0]\}$ between two intersections
of $\yashar$ with the real projective hypersurface $\cV_P(\RR)$.
If $Q$ does not contain $X_0$ as a factor, we can consider instead of $\yashar \setminus \{[X^0]\}$
the two open rays $\yashar_\pm$ starting at $x^0$
of a general straight line through $x^0$ in $\RR^d$ and their intersections with the real affine hypersurfaces
$\cV_\qut(\RR)$ and $\cV_p(\RR)$. 

An example of a polynomial $Q$ interlacing $P$ is the first
directional derivative $P^{(1)}_{x^0}$, see \eqref{eq:renegar_der} (in this case $\qut=p^{(1)}_{x^0}$
is the first Renegar derivative).

It is not hard to see that (upon normalizing $p(x^0)>0$)
the definition of interlacing is independent of the choice of a point $x^0$ in a rigidly convex algebraic interior
with a minimal defining polynomial $p$.
In case the real projective hypersurfaces $\cV_P(\RR)$ and $\cV_Q(\RR)$ are both smooth, the interlacing of polynomials
simply means the interlacing of ovaloids, see Proposition \ref{prop:oval}.
More precisely, in this case $Q$ interlaces $P$ if and only if
\begin{itemize}
\item[a.]
If $m=2k$ is even and $\cV_P(\RR) = W_1 \coprod \cdots \coprod W_k$ and $\cV_Q(\RR)=W_1' \coprod \cdots \coprod W_k'$
are the decompositions into connected components,
then the ovaloid 
$W_i'$ is contained in the ``shell'' obtained by removing the interior of 
the ovaloid $W_i$ from the closure of the interior of the ovaloid $W_{i+1}$, $i=1,\ldots,k-1$,
and the pseudo-hyperplane $W_k'$ is contained in the closure of the exterior of the ovaloid $W_k$;
\item[b.]
If $m=2k+1$ is odd and $\cV_P(\RR) = W_1 \coprod \cdots \coprod W_k \coprod W_{k+1}$ and 
$\cV_Q(\RR)=W_1' \coprod \cdots \coprod W_k'$
are the decompositions into connected components,
then the ovaloid 
$W_i'$ is contained in the ``shell'' obtained by removing the interior of 
the ovaloid $W_i$ from the closure of the interior of the ovaloid $W_{i+1}$, $i=1,\ldots,k-1$,
and the ovaloid $W_k'$ is contained in the closure of the exterior of the ovaloid $W_k$
and the pseudo-hyperplane $W_{k+1}$ is contained in the closure of the exterior of $W_k'$.
\end{itemize}

Interlacing can be tested via the Bezoutiant, similarly to testing the $RZ$ condition via
the Hermite matrix. 
For polynomials $f, g \in \RR[t]$ with $f$ of degree $m$ and $g$ of degree at most $m$, 
we define the Bezoutiant of $f$ and $g$, $B(f,g)=\left[b_{ij}\right]_{i,j=1,\ldots,m}$,
by the identity
$$
\frac{f(t)g(s)-f(s)g(t)}{t-s} = \sum_{i,j=0}^{m-1} b_{ij}t^is^j;
$$
notice that the entries of $B(f,g)$ are polynomials in the coefficients of $f$ and of $g$. 
The nullity of $B(f,g)$ equals the number of common zeroes of $f$ and of $g$
(counting multiplicities), and (assuming that the degree of $g$ is at most $m-1$),
$B(f,g)>0$ if and only if $f$ has only real and distinct zeroes and  
there is a zero of $g$ in the open interval
between any two zeroes of $f$; see, e.g., \cite{KN36}.
Given $p \in \RR[x_1,\ldots,x_d]$ a reduced polynomial of degree $m$ with $p(x^0) \neq 0$, with homogenization $P$,
and $Q \in \RR[X_0,X_1,\ldots,X_d]$ a homogeneous polynomial of degree $m-1$
that is relatively prime with $P$, we now consider $B(\check p_x,\check \qut{}_x)$,
where $\check p_x$, $\check \qut{}_x$ are as before;
it is a polynomial matrix that we call the Bezoutiant of $P$ and $Q$ with respect to $x^0$
and denote $B(P,Q;x^0)$.
We see that $p$ is a $RZ_{x^0}$ polynomial and $Q$ interlaces $P$ if and only if 
$B(P,Q;x^0)(x) \geq 0$ for all $x \in \RR^d$.

\subsection{}

Before stating and proving the main result of this section, we make some preliminary observations.

Let $P \in \CC[X_0,X_1,\ldots,X_d]$ be a reduced homogeneous polynomial of degree $m$ with the corresponding
complex projective hypersurface $\cV_P$ (see \eqref{eq:proj_hyper}, and let $U$ be a determinantal representation
of $P$ with the adjoint matrix $V$ as in \eqref{eq:U_and_V}. Since $\dim \ker U(X) = 1$ for a general point $[X]$
of any irreducible component of $\cV_P$, the rows of $V$ are proportional along $\cV_P$ and so are the columns.
An immediate consequence is that no element of $V$ can vanish along $\cV_P$: otherwise, 
because of the proportionality of the rows, a whole row or a whole column of $V$ would vanish along $\cV_P$,
hence be divisible by $P$, hence be identically $0$ (since all the elements have 
degree $m-1$ which is less than the degree of $P$),
implying that $\det V$ is identically $0$, a contradiction.
Another consequence is that every minor of order $2$ in $V$, $V_{ij}V_{kl} - V_{kj} V_{il}$, vanishes along $\cV_P$.

\begin{lem} \label{lem:pairing}
Let $F_j = \left[V_{ij}\right]_{i=1,\ldots,m}$, $j=1,\ldots,m$, and $G_i = \left[V_{ij}\right]_{j=1,\ldots,m}$,
$i=1,\ldots,m$, be the columns and the rows of the adjoint matrix $V$, respectively,
let $X^0=(X^0_0,X^0_1,\ldots,X^0_d) \in \CC^{d+1} \setminus \{0\}$,
and let 
\begin{equation} \label{eq:dirder}
P'_{X^0}(X) = \frac{d}{ds} \left. P(X + s X^0) \right|_{s=0} 
= \sum_{\alpha=0}^d X^0_\alpha \, \frac{\partial P}{\partial X_\alpha}(X)
\end{equation}
be the directional derivative.
Then
\begin{equation} \label{eq:pairing}
G_i \, U(X^0) \, F_j = V_{ij} \, P'_{X^0}
\end{equation}
along $\cV_P$.
\end{lem}

The result follows immediately by substituting \eqref{eq:derdetrep} into \eqref{eq:dirder}
to calculate the directional derivative in terms of the entries of the adjoint matrix and of the coefficient
matrices of the determinantal representation, and using the vanishing of the minors of order $2$ in $V$
along $\cV_P$. A version of \eqref{eq:pairing} was established in \cite[Corollary 5.8]{Vin93} in case $d=2$
and $\cV_P$ is smooth (the proof given there works verbatim for fully saturated determinantal representations,
see Section \ref{subsec:sing},
when $\cV_P$ is possibly singular and / or reducible) using essentially the pairing between the kernel
and the left kernel alluded to in Section \ref{subsec:theta}.

Assume now that the dehomogenization $p(x_1,\ldots,x_d)=P(1,x_1,\ldots,x_d)$ is a $RZ_{x^0}$ polynomial
with $p(x^0) \neq 0$,
let $X^0=(1,x^0)$, and let $U$ be a self-adjoint determinantal representation.
Let $\yashar$ be a straight line through $[X^0]$ in $\PP^d(\RR)$ intersecting $\cV_P(\RR)$ in $m$
distinct points $[X^1],\ldots,[X^m]$. Then we have

\begin{lem} \label{lem:intersect}
$U(X^0) > 0$ if and only if the compression of $U(X^0)$ to $\ker U(X^i)$ is positive definite for $i=1,\ldots,m$.
\end{lem}

This is just a special case of \cite[Proposition 5.5]{Vin93}: the statement there is for $d=2$ but the proof
for general $d$ is exactly the same 
(it ammounts to restricting the determinantal representation $U$ to the straight line $\yashar$,
and looking at the canonical form of the resulting hermitian matrix pencil).
We give a direct argument in our situation.
 
\begin{proof}[Proof of Lemma \ref{lem:intersect}]
Choose $X \in \RR^{d+1}$ so that $\yashar \setminus \{[X^0]\} = \{[X-sX^0]\}_{s \in \RR}$. Then $X^i=X-s_iX^0$,
where $s_i$, $i=1,\ldots,m$, are the zeroes of the univariate polynomial $P(X-sX^0)$, i.e.,
the eigenvalues of the generalized eigenvalue problem
$$
\left(U(X)-sU(X^0)\right) v = 0.
$$
The corresesponding eigenspaces are precisely $\ker U(X^i)$; since there are $m$ distinct eigenvalues, these
eigenspaces span all of ${\mathbb C}^m$,
$$
{\mathbb C}^m = \ker U(X^1) \dot + \cdots \dot + \ker U(X^m).
$$
The lemma now follows since the different eigenspaces are orthogonal with respect to $U(X^0)$:
if $v_i \in \ker U(X^i)$, $v_j \in \ker U(X^j)$, $i \neq j$, then
$$
s_i v_j^* U(X^0) v_i = v_j^* U(X) v_i = s_j v_j^* U(X^0) v_i
$$
(since $s_j \in \RR$), implying that $v_j^* U(X^0) v_i = 0$ (since $s_i \neq s_j$).
\end{proof}

We notice that Lemma \ref{lem:intersect} remains true for non-reduced polynomials $P$ provided the determinantal 
representation $U$ is {\em generically maximal} (or {\em generically maximally generated}) \cite{KV1}:
if $P=P_1^{r_1} \cdots P_k^{r_k}$, where $P_1,\ldots,P_k$ are distinct irreducible polynomials, this means that
that $\dim \ker U(X) = r_i$ at a general point $[X]$ of $\cV_{P_i}$, $i=1,\ldots,k$.
Since positive self-adjoint determinantal representations are always generically maximal, this may open the possibility
of generalizing Theorem \ref{thm:interlace} below to the non-reduced setting.

\begin{thm} \label{thm:interlace}
Let $p \in \RR[x_1,\ldots,x_d]$ be an irreducible $RZ_{x^0}$ polynomial of degree $m$ with $p(x^0) \neq 0$,
let $P$ be the homogenization of $p$, and let $X^0=(1,x^0)$. 
Let $U$ be a self-adjoint determinantal representation of $P$ with adjoint matrix $V$,
as in \eqref{eq:U_and_V}. Then $U(X^0)$ is either positive or negative definite if and only if
the polynomial $V_{jj}$ interlaces $P$; here $j$ is any integer between $1$ and $m$.
\end{thm}

\begin{proof}
The fact that $U(X^0)>0$ implies the interlacing follows immediately 
from Cauchy's interlace theorem for eigenvalues of Hermitian matrices, see, e.g, \cite{Hw04}.
We provide a unified proof for both directions.

Let $\yashar$ be a straight line through $[X^0]$ in $\PP^d(\RR)$ intersecting $\cV_P(\RR)$ in $m$ distinct points
$[X^1],\ldots,[X^m]$ none of which is a zero of $V_{jj}$. Lemma \ref{lem:pairing} implies that for any $[X] \in \cV_P(\RR)$,
$$
F_j(X)^* \, U(X^0) \, F_j(X) = P'_{X^0}(X) \, V_{jj}(X).
$$
Lemma \ref{lem:intersect} then shows that $U(X^0)$ is positive or negative definite if and only if
$P'_{X^0} V_{jj}$ has the same sign (positive or negative, respectively) at $X^i$ for $i=1,\ldots,m$.

Similarly to the proof of Lemma \ref{lem:intersect}, let us choose
$X \in \RR^{d+1}$ so that $\yashar \setminus \{[X^0]\} = \{[X+sX^0]\}_{s \in \RR}$, so that $X^i=X+s_iX^0$,
where $s_1 < \cdots < s_m$
are the zeroes of the univariate polynomial $P(X+sX^0)$.
It follows from Rolle's Theorem that
$\dfrac{d}{ds} P(X+sX^0) = P'_{X^0}(X+sX^0)$ has exactly one zero in each open interval $(s_i,s_{i+1})$,
$i=1,\ldots,m-1$, hence has opposite signs at $s_i$ and at $s_{i+1}$. Therefore
$U(X^0)$ is positive or negative definite if and only if $V_{jj}(X+sX^0)$ 
has opposite signs at $s_i$ and at $s_{i+1}$, i.e., if and only if $V_{jj}$ interlaces $P$.
\end{proof}  

It would be interesting to find an analogue of Theorem \ref{thm:interlace} for other signatures
of a self-adjoint determinantal representation, similarly to \cite[Section 5]{Vin93}.

Combining Theorem \ref{thm:interlace} with the construction of determinantal representations that was
sketched in Section \ref{subsec:divisors} (see also Section \ref{subsec:sing} for the extension
of the construction to the singular case) then yields the following result.

\begin{thm} \label{thm:new_2d}
Let $p \in \RR[x_1,x_2]$ be an irreducible $RZ_{x^0}$ polynomial of degree $m$ with $p(x^0)=1$,
let $P$ be the homogenization of $p$, let $\nu \colon \tcV_P \to \cV_P$ be the desingularization
of the corresponding complex projective curve, and let $\Delta$ be the adjoint divisor on $\tcV_P$.
Let $Q \in \RR[X_0,X_1,X_2]$ be a homogeneous polynomial of degree $m-1$ that interlaces $P$
and that vanishes on the adjoint divisor: $(\nu^* Q) \geq \Delta$.
Then there exist $A_0,A_1,A_2 \in \herm\mat{{\mathbb C}}{m}$ with $A_0+x^0_1A_1+x^0_2A_2=I$
such that $\det(A_0+x_1A_1+x_2A_2)=p(x)$ and such that the first principal minor of $A_0+x_1A_1+x_2A_2$
equals $Q(1,x_1,x_2)$.
\end{thm}

We emphasize that the determinantal representation $A_0+x_1A_1+x_2A_2$ is given by an explicit algebraic construction
starting with $P$ and $Q$.
Theorem \ref{thm:new_2d} implies a version of Theorem \ref{thm:2d} for positive self-adjoint determinantal representations
since there certainly exist interlacing polynomials vanishing on the adjoint divisor:
we can take the directional derivative $Q = P^{(1)}_{x'}$ for any interior point $x'$ of the rigidly convex algebraic 
interior containing $x^0$ with a minimal defining polynomial $p$.
The two basic open questions here are:
\begin{enumerate}
\item
``How many'' positive self-adjoint determinantal representations does one obtain starting with directional derivatives
as above?
\item
What other methods are there to produce interlacing polynomials (vanishing on the adjoint divisor)?
\end{enumerate}

\begin{proof}[Proof of Theorem \ref{thm:new_2d}]
It is not hard to see that $Q$ interlacing $P$ implies that $\cV_Q$ is contact to $\cV_P$ at real points of intersection,
and that we can write $(Q) = D + D^\tau + \Delta$. It only remains to show that $D-(L)$ is not linearly equivalent to 
an effective divisor, where $L$ is a linear form. 

Notice that $\tau$ lifts to an antiholomorphic involution on the desingularization (this was already implied
when we wrote, e.g., $D^\tau$). Furthermore, the fact that $p$ is a $RZ$ polynomial, implies that 
$\tcV_P$ is a compact real Riemann surface of dividing type, i.e., $\tcV_P \setminus \tcV_P(\RR)$ consists
of two connected components interchanged by $\tau$, where $\tcV_P(\RR)$ is the fixed point set of $\tau$,
see \cite{Vin93} and the references therein and \cite{HV07}.
We orient $\tcV_P(\RR)$ as the boundary of one of these two connected components.

It is now convenient to change projective coordinates so that $[X^0]=[1,x^0]$ becomes $[0,0,1]$.
It is not hard to see that in the new coordinates, both the meromorphic differential $\nu^* dx_1$ and the function
$\nu^* \dfrac{Q(1,x_1,x_2)}{\partial p/\partial x_2}$ have constant sign (are either everywhere nonnegative or everywhere
nonpositive) on $\tcV_P(\RR)$.
It follows that so is the meromorphic differential $\omega = \nu^* \dfrac{Q(1,x_1,x_2) dx_1}{\partial p/\partial x_2}$.
We have (see, e.g., \cite[Appendix A2]{ACGH85}) $(\omega)=(Q) - \Delta -2(X_0) = D+D^\tau-2(X_0)$. 
If there existed a rational function $f$ and an effective
divisor $E$ on $\tcV_P$ so that $(f) + D - (X_0) = E$, we would have obtained that $(f \omega f^\tau) = E + E^\tau$,
i.e., $f \omega f^\tau$ is a nonzero holomorphic differential that is everywhere nonnegative or everywhere nonpositive
on $\tcV_P(\RR)$, a contradiction since its integral over $\tcV_P(\RR)$ has to vanish by Cauchy's Theorem.  
\end{proof} 

Notice that this proof is essentially an adaptation of \cite[Proposition 4.2]{Vin93} which is itself an adaptation
of \cite{Fay73}; it would be interesting to find a more elementary argument.

\end{document}